\documentclass[a4paper,12pt]{article}

\usepackage{color}
\usepackage[pdftex]{graphicx}
\usepackage{amsmath,amssymb,amsthm}
\usepackage{wrapfig}
\usepackage{ifpdf}
\usepackage{hyperref}

\newtheorem{theorem}{Theorem}

\newtheorem{lemma}{Lemma}

\title {Minimal simplicial spherical mappings with a given degree}

\author{Ksenia Apolonskaya and Oleg R. Musin}

\begin{document}
\date{}
\maketitle
	
%\begin{center}

%\end{center}

\begin{abstract} 
This paper studies the minimal number of vertices $\lambda(n,d)$ required in a triangulation of the $n$-sphere to admit a simplicial map to the boundary of a $(n+1)$-simplex with a given degree $d$. We establish upper bounds for $\lambda(n,d)$ in dimensions $n \geq 3$.   Furthermore, we provide exact formulas for small values of $d$, showing that $\lambda(n,d)=n+d+3$ for $n \geq 3$ and $d=2,3,4$. A key technical result is the identity $\lambda(n,d) = \lambda(d-1,d) + n - d + 1$ for $n \geq d$, which allows us to reduce higher-dimensional cases to lower-dimensional ones. The proofs involve constructive methods based on local modifications of triangulations and combinatorial arguments. 
\end {abstract}

\section{Introduction}
The degree of a map between oriented $n$--dimensional manifolds was first defined by Brouwer in 1911. The Hopf theorem states that the {degree of a continuous map} is a topological invariant (see, for instance, \cite{Milnor} and \cite[pp. 44--46]{Mat}). Let $f:M_1\to M_2$ be a smooth map and $y\in M_2$ be a regular value for $f$. Then 
\[
    \deg f = \sum_{x \in f^{-1}(y)} \text{sign}(J_f(x)), \eqno (1)
    \]
    where $J_f(x)$ is the Jacobian of $f$ at point $x$, and $\text{sign}(J_f(x))$ equals $+1$ or $-1$ depending on whether $f$ preserves its orientation at $x$ or reverses it.

There is a simplicial version of (1) and its extension for simplicial maps $f:S^m\to S^n$, see \cite{Mad,MS,MS2,Mus14,Mus16,MusQS}. Let $T_1$ and $T_2$ be triangulations of $M_1$ and $M_2$. Let $f:V(T_1)\to V(T_2)$ be a simplicial map, where $V(T)$ denotes the set of vertices of $T$. Consider an $n$-simplex $s$ of $T_2$ and denote by $\Pi(s)$ the set of preimages of $s$ in $T_1$. Then every $t\in \Pi_f(s)$ is a $n$--simplex in $T_1$. We have $f(t)=s$, $H:=f|_t:t\to s$ is a linear map, and $\det(H)\ne0$. Denote by $\text{sign}(f_t)$ the $\text{sign}$ of $\det(H)$. It can be proved that the sum of all $\text{sign}(f_t)$, $t\in \Pi_f(s)$,  does not depend on $s$ and  
\[
    \deg f = \sum_{t \in \Pi_f(s)} \text{sign}(f_t). \eqno (2)
    \]

In this paper, we consider the case where $M_1$ and $M_2$ are $n$--spheres. Let $T_1$ be a triangulation of  $S^n$ and $T_2$ be the boundary triangulation of $(n+1)$--simplex $\Delta^{n+1}$. We denote it by $S_{n+2}^n$. % that is $\partial \Delta^{n+1} \) denote the boundary of an \((n+1)\)-simplex, which is also a triangulation of \( S^n \). 
Let $$L:V(T_1) \to V(S_{n+2}^n)=\{1,...,n+2\}$$ be a coloring (labeling) of the vertex set of $T_1$. Then $L$ can be extended linearly to all $k$--dimensional, $k\ge 1$, faces of $T_1$ and we obtain a simplicial map 
$$f_L: {S}^n \rightarrow {S}_{n+2}^n.$$
{\em For any integer $d$, let $\lambda(n,d)$ denote the smallest number of vertices that a triangulation $T_1$ of $n$--sphere must have in order for there to exist a coloring $L:V(T_1) \to V(S_{n+2}^n)$ with $\deg f_L=d$.} %Let 
%$$\alpha(n):=\limsup\limits_{d \to \infty}{\frac{\lambda(n,d)}{d}}.$$

\begin{figure}[h]
    \centering
    \includegraphics[height = 7cm]{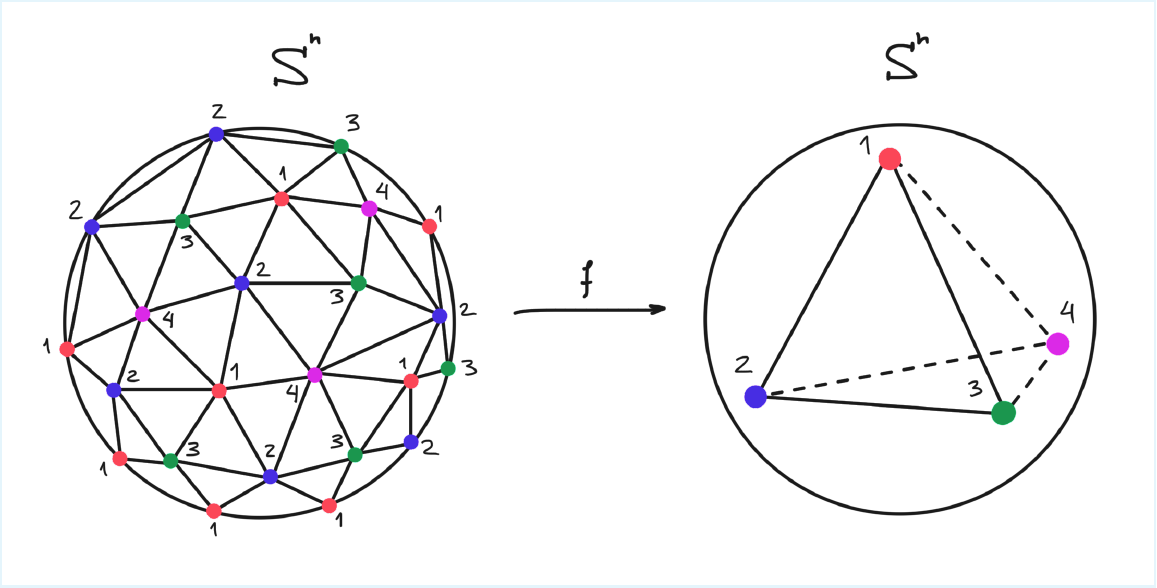}
    \caption{Simplicial mapping between $T_1$ and $\partial \Delta^{n+1}$.}
    \label{fig:1}
\end{figure}

For all $n$ we obviously have 
$$\lambda(n,-d)=\lambda(n,d), \quad \lambda(n,1)=n+2.$$ 

It is clear that if $d>0$ then $\lambda(1,d)=3d$. Indeed, at the vertices of the polygon with $3d$ vertices, we place the labels 123123...123 in a cyclic order.

The case $n=2$ is completely solved by Madahar and  Sarkaria \cite{MS2}. They proved that 
$$
\lambda(2,d)=2|d|+2, \; |d|\ge 3. $$
Recently, some bounds for $\lambda(n,d)$ have been received in \cite{Basak}, in particular 
 $$\lambda(n,d)\le d + n + 3, \quad d\le n.$$ 
 We show that this inequality is equality for $d=2,3,4$, see Theorem 3.

\medskip 

Our main results are the following theorems.
\begin{theorem}
\label{thm:1}
$\lambda(n,d)\leq \frac{n+2}{n}d+(2n+2) $. %In particular, $\alpha(n)\le \frac{n+2}{n}$.
\end{theorem}
%\medskip

%Moreover, we have the following equalities: 

%\medskip 

\begin{theorem}
\label{thm::2}
$\lambda(n,d)=n+\lambda(d-1,d)-d+1$ for $n\geq d$.% and $\beta(d)=1$,where
%$$
%\beta(d):=\lim\limits_{n \to \infty}{\frac{\lambda(n,d)}{n}}.
%$$  
\end{theorem}

%\medskip
\begin{theorem}
\label{thm::3}
$\lambda(n,d)=n+d+3$ for $n\geq d-1$ and $d=2,3,4$.
\end{theorem}

\medskip

Proofs of these theorems are given in the next section. Several problems remain here. One of the interesting questions is the following 

\medskip

\noindent {\bf Problem.} {\em Find the minimum $\omega$ such that}
$$\limsup\limits_{d \to \infty}{\frac{\lambda(n,d)}{d^{\,\omega}}}<\infty.$$
Recently, Ryabichev \cite{Ryab} proved that $\omega<1$.

\section{Proof of the main theorems}

\medskip

\subsection{Auxiliary lemmas.}
\begin{lemma}
\label{lem:1}
$\lambda(n+1,d)\leq \lambda(n,d)+1$
\end{lemma}
\begin{proof}
Let \( K \) be a triangulation of \( S^n \) with \( \lambda(n,d) \) vertices. Arbitrarily triangulate the interior of \( K \). Add a new vertex \( x \) and form the join \( K * x \) with the boundary of \( K \). Assign a new color to the vertex \( x \) (see Figure~\ref{fig:10}).

This construction yields a triangulation \( K' \) of \( S^{n+1} \) with \( \lambda(n,d) + 1 \) vertices. The degree remains unchanged: a positively oriented simplex in \( K \) extends to a positively oriented simplex in \( K' \) when joined with \( x \), and similarly for negatively oriented simplices.

\end{proof}
\begin{figure}[h!]
    \centering
    \includegraphics[height = 5.5cm]{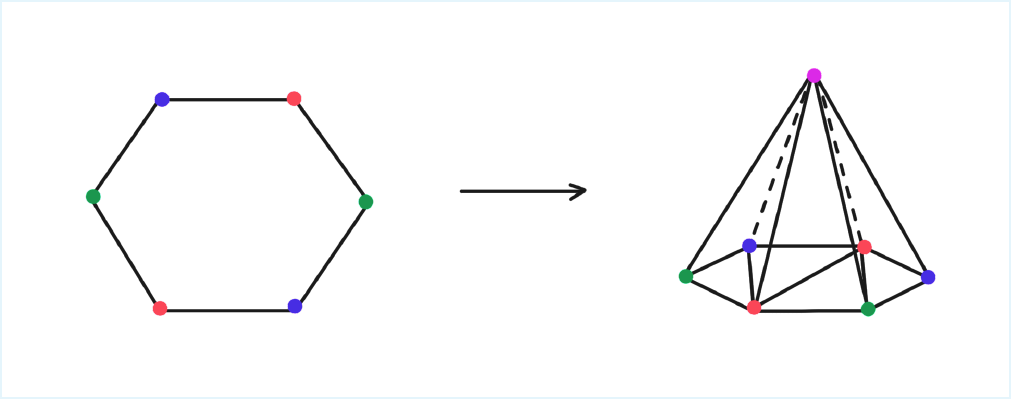}
    \caption{$\lambda(2,d)\le \lambda(1,d)+1$}
    \label{fig:10}
\end{figure}

\begin{lemma}
\label{lem:2}
If $\lambda(n, d)\leq 2n+3$, then  $\lambda(n, d)\geq \lambda(n-1,d)+1$
\end{lemma}
\begin{proof}
The simplex $\Delta^{n+1}$ has $n+2$ vertices. Suppose the triangulation of the second sphere has at most $2n+3$ vertices. Then, by the pigeonhole principle, some vertex has exactly one preimage. 
Removing this vertex together with the interior of the remaining disk reduces the dimension by 1 while preserving the degree.
\end{proof}
\medskip
\begin{lemma}
\label{lem:3}
If $n\geq d-1$, then $\lambda(n,d)\geq \lambda(d-1, d)+n-d+1$
\end{lemma}
\begin{proof}
We know that for $n\geq d$, $\lambda(n,d)\leq n+d+3$ from Article \cite{Basak}. Then $\lambda(n,d)\leq n+d+3 \leq 2n+3$, and from the previous statement it follows that $\lambda(n,d)\geq \lambda(n-1,d)+1$. Similarly, descending to $n=d$, we obtain that $\lambda(n,d)\geq \lambda(d-1,d)+n-d+1$. 
\end{proof}
\medskip
\begin{lemma}
\label{lem:4}
$\lambda(3,4)=10$
\end{lemma}
\begin{proof}
We first prove that $\lambda(3,4) \geq 10$.\\
The simplex $\Delta^4$ has five vertices. Let $U_i$ denote the set of preimages of vertex $i$, for $i=1,\dots,5$. If $|U_i| \geq 2$ for all $i$, then $\lambda(3,4) \geq 10$.

Now suppose $|U_1| = 1$. Remove $U_1$ together with all vertices not connected to it and the corresponding edges. The result is a triangulation of $S^2$ with degree $4$, which requires at least $\lambda(2,4) = 10$ vertices. Adding back $U_1$ gives at least $11$ vertices.

We now construct an example with exactly $10$ vertices. Start with the standard simplex $u_1 u_2 u_3 u_4 u_5$. Consider the face $u_1 u_2 u_3 u_4$ and insert vertex $w_5$, replacing it with simplices $u_1 u_2 u_3 u_5, \dots, u_2 u_3 u_4 u_5$. Repeat this process for all other faces, adding vertices $w_1, w_2, w_3, w_4$.

The resulting triangulation has $10$ vertices and four preimages of $v_1 v_2 v_3 v_4$ (namely $w_1 v_2 v_3 v_4$, $v_1 w_2 v_3 v_4$, $v_1 v_2 w_3 v_4$, $v_1 v_2 v_3 w_4$), all negatively oriented. Permuting the vertex ordering in $v_1 v_2 v_3 v_4 v_5$ yields degree $d=4$.

Thus, $\lambda(3,4) \geq 10$ and $\lambda(3,4) \leq 10$, so $\lambda(3,4) = 10$.
\end{proof}

\subsection{Proof of Theorem 1.}
\begin{proof}

We consider the number \( d = kn + l \), where \( k \in \mathbb{Z}_{\geq 0} \), \( 0 \leq l < n \). In the case \( k=0 \), we have \( d = l \).

The theorem was proved by Basak, Gupta, and Trivedi \cite{Basak}, but we provide an alternative proof.
 
 \medskip

\noindent {\bf i. Construction of the example.}

Initially, take the simplex \( u_1 u_2 \dots u_{n+2} \). Insert a new vertex \( w_{n+2} \) inside the face \( u_1 \dots u_{n+1} \). Now the map \( f \) has degree \( \deg f = 0 \). Recall that \( f \) maps each vertex \( u_i, w_i \) to \( v_i \) on the boundary \( \partial \Delta^{n+2} \).

If \( l = 0 \), we stop here. Otherwise, add vertices \( w_1, w_2, \dots, w_l \), inserting each into the corresponding simplex. This creates \( l \) positively oriented preimages, so \( d = l \). The number of vertices at this stage is \( n + 3 + l \).

\begin{figure}[h]
    \centering
    \includegraphics[height=5cm]{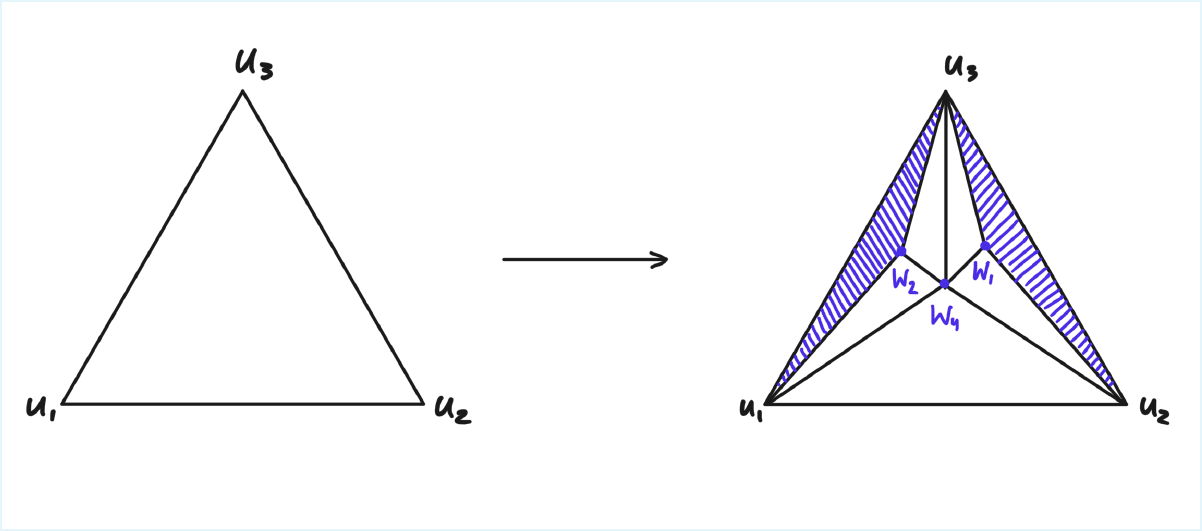}
    \caption{Case \( n=2, l=2 \)}
    \label{fig:step1}
\end{figure}

 \medskip

\noindent {\bf ii. Increasing degree and vertices for \( k > 0 \).}

For \( k > 0 \), we increase the degree by \( n \) and the number of vertices by \( n+2 \) at each step, repeating the operation \( k \) times.

Consider a positively oriented simplex in the triangulation:
\[
u = u_1 u_2 \dots u_{n+1}.
\]

Insert a vertex \( w_{n+2} \) inside \( u \). Define
\[
a_k = u_1 \dots u_{k-1} u_{k+1} \dots u_{n+1} w_{n+2}, \quad 1 \leq k \leq n+1.
\]
This operation replaces the simplex \( u \) by \( n+1 \) simplices \( a_1, a_2, \dots, a_{n+1} \).

Next, insert vertices \( w_i \) inside each \( a_i \) for \( 1 \leq i \leq n+1 \). Thus, one positive simplex is replaced by \( n+1 \) positive simplices.

This results in \( n+1 \) positive preimages of \( v_1 \dots v_{n+1} \), so the degree increases by \( n \), and the number of vertices increases by \( n+2 \).

\begin{figure}[h]
    \centering
    \includegraphics[height=5cm]{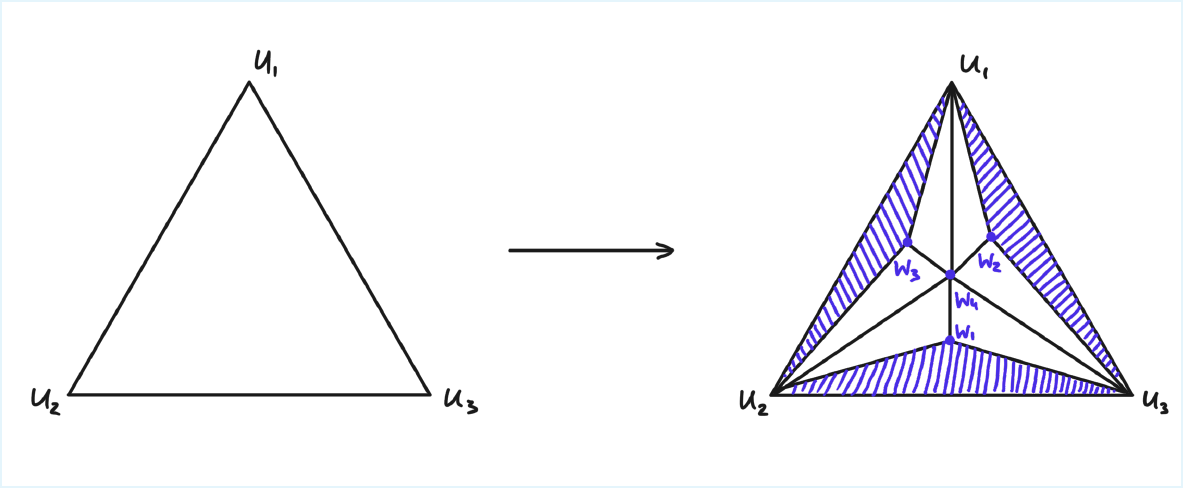}
    \caption{One step of the process for \( n=2 \). Number of vertices increases by 4; degree increases by 2}
    \label{fig:step2}
\end{figure}

 \medskip

\noindent {\bf iii. Final vertex count.}

Performing this operation \( k \) times, each time finding a positively oriented simplex and replacing it by inserting \( n+2 \) new vertices:

- The initial degree is \( l \leq n \).
- The initial number of vertices is \( n + 3 + l \).
- After \( k \) steps, the number of vertices increases by \( k(n+2) \).

Thus, for \( d = kn + l \), the total number of vertices satisfies
\[
n + 3 + l + k(n+2) \leq (n+2) k + 2n + 2.
\]
Using \( d = kn + l \), we have the estimate
\[
|V| \leq \frac{n+2}{n} d + 2n + 2.
\]
\end{proof}

\subsection{Proof of Theorem 2.}
\begin{proof}
From Lemma~\ref{lem:1}, we have $\lambda(n,d) \leq \lambda(n-1,d) + 1$. Iterating this inequality gives
\[
\lambda(n,d) \leq \lambda(d-1,d) + n - d + 1.
\]
The construction in the previous section shows $\lambda(n,d) \geq \lambda(d-1,d) + n - d + 1$. Therefore,
\[
\lambda(n,d) = \lambda(d-1,d) + n - d + 1.
\]
\end{proof}

%\medskip

\subsection{Proof of Theorem 3.}
\begin{proof}
By Theorem~\ref{thm::2}, we have $\lambda(n,d) = \lambda(d-1,d) + n - d + 1$.

The base cases are $\lambda(1,2) = 6$, $\lambda(2,3) = 8$, and $\lambda(3,4)$ from Lemma~\ref{lem:4}. Thus,
\begin{align*}
\lambda(n,2) &= n + 5, \\
\lambda(n,3) &= n + 6, \\
\lambda(n,4) &= n + 7.
\end{align*}
\end{proof}

 \bigskip

\noindent {\bf {Acknowledgments.}}
Ksenia Apolonskaya is supported by the ``Priority 2030'' strategic academic leadership program. The author thanks the Summer Research Programme at MIPT - LIPS-25 for the opportunity to work on this and other problems.%(Phystech School of applied mathematics and computer science).\\

 \bigskip 
 
 K. Apolonskaya, Saint Petersburg State University, Dept. of Mathematics and Computer Science
 
 {\it E-mail address:} apolonskayaks@gmail.com
 
 \medskip
 O. R. Musin,  University of Texas Rio Grande Valley, School of Mathematical and
 Statistical Sciences, One West University Boulevard, Brownsville, TX, 78520, USA.

 {\it E-mail address:} oleg.musin@utrgv.edu

%\section*{References}
%$[1]$ \hypertarget{musin}{} Homotopy groups and quantitative Sperner–type lemma
%Oleg R. Musin\\
%\\
%$[2]$ \hypertarget{basak}{} Simplicial degree d self-maps on n-spheres
%Biplab Basak, Raju Kumar Gupta, and Ayushi Trivedi (2024)\\

%$[3]$ \hypertarget{madahar}{} Minimal Simplicial Self-maps of the 2-Sphere
%K.V. Madahar and K. S. Sarkaria (1999)\\
%\\

\end{document}